\documentclass[11pt]{amsart}

\usepackage{amsthm, amsfonts, amssymb, amscd}
\usepackage[pagebackref,colorlinks]{hyperref}
\usepackage{tikz-cd}
\usepackage{geometry}
\usepackage{marginnote}

\theoremstyle{definition}
\newtheorem{ntn}{Notation}[section]

\theoremstyle{plain}
\newtheorem{lem}[ntn]{Lemma}
\newtheorem{prp}[ntn]{Proposition}
\newtheorem{thm}[ntn]{Theorem}
\newtheorem{cor}[ntn]{Corollary}

\theoremstyle{definition}

\numberwithin{equation}{section}

\newcommand{\z}{\mathbb{Z}}
\newcommand{\q}{\mathbb{Q}}

\newcommand{\EE}{\mathcal{E}}

\newcommand{\BB}{\mathcal{B}}
\newcommand{\GG}{\mathcal{G}}
\newcommand{\RR}{\mathcal{R}}
\newcommand{\WW}{\mathcal{W}}
\newcommand{\II}{\mathcal{I}}
\newcommand{\RP}{\mathcal{RP}}
\newcommand{\PP}{\mathcal{P}}
\newcommand{\RB}{\mathcal{RB}}
\newcommand{\TT}{\mathcal{T}}

\renewcommand{\SS}{\mathcal{S}}

\newcommand{\mmm}{\mathfrak{m}}

\renewcommand{\aa}{{A^\times}}

\newcommand{\tors}{{{\rm Tor}_1^{\z}}}

\newcommand{\Lan}{\langle\! \langle}
\newcommand{\Ran}{\rangle \!\rangle}

\newcommand{\mt}{\mapsto}
\newcommand{\lan}{\langle}
\newcommand{\ran}{\rangle}
\newcommand{\se}{\subseteq}
\newcommand{\arr}{\rightarrow}
\newcommand{\larr}{\longrightarrow}
\newcommand{\harr}{\hookrightarrow}
\newcommand{\two}{\twoheadrightarrow}

\newcommand{\stabe}{{\rm Stab}}
\newcommand{\GL}{\mathit{{\rm GL}}}

\newcommand{\SL}{\mathit{{\rm SL}}}

\newcommand{\Ind}{{\rm Ind}}

\newcommand{\ind}{{\rm ind}}

\newcommand{\inc}{{\rm inc}}
\newcommand{\id}{{\rm id}}

\newcommand {\mtx}[2]
{\left(\!\!\!
	\begin{array}{cc}
		#1    \\
		#2 
	\end{array}
	\!\!\!\right)}

\newcommand {\mtxx}[4]
{\left(\!\!
	\begin{array}{cc}
		\!\!#1 & \!\!#2   \\
		\!\!#3 & \!\!#4
\end{array}\!\!
	\right)
}

\newtheoremstyle{athm}
{}
{}
{\itshape}
{}
{\scshape}
{}
{.5em}
{\thmnote{#3}}
\theoremstyle{athm}

\begin{document}
\title{A refined Bloch-Wigner exact sequence in characteristic 2}
\author{Behrooz Mirzaii}
\author{Elvis Torres P\'erez}
\begin{abstract}
Let $A$ be a local domain of characteristic $2$ such that  its residue field has more than $64$ elements. Then we find an exact relation between
the third integral homology of the group $\SL_2(A)$ and Hutchinson's refined Bloch group $\RB(A)$. 
%In fact we prove that there is a refined Bloch-Wigner exact sequence 
%\[
%0\arr \tors(\mu(A),\mu(A)) \arr H_3(\SL_2(A),\z) \arr \RB(A)\arr 0. 
%\]
\end{abstract}
\maketitle

%%%%%%%%%%%%%%%%%%%%%%%%%%%%%%%%%%%%%%%%%%%%%%%%%%%%%%%%%%%%%%
\section*{Introduction}
%%%%%%%%%%%%%%%%%%%%%%%%%%%%%%%%%%%%%%%%%%%%%%%%%%%%%%%%%%%%%%

The classical Bloch-Wigner exact sequence studies the indecomposable part of the third $K$-group of a field \cite{dupont-sah1982}. 
The general Bloch-Wigner exact sequence for fields, claims that for any field $F$ we have the exact sequence
\[
0 \arr \tors(\mu(F), \mu(F))^\sim \arr K_3^\ind(F)\arr \BB(F) \arr 0.
\]
Here $\BB(F)$, called the Bloch group of $F$, is a certain subgroup of the scissors congruence group $\PP(A)$ (see Section \ref{sec1}) 
and $\tors(\mu(F), \mu(F))^\sim$ is the unique nontrivial extension of  $\mu_2(F)$ by  $\tors(\mu(F), \mu(F))$ \cite{suslin1991}, 
\cite{hutchinson2013}.  This exact sequence can be extended to local domains, where its residue  field has more than $9$ elements 
\cite{mirzaii2017}.

When $F$ is quadratically closed, we have a natural isomorphism $H_3(\SL_2(F), \z)\simeq K_3^\ind(F)$ \cite{mirzaii-2008},\cite{sah1989}.  
In general we have a natural surjective map
\[
H_3(\SL_2(F),\z) \two K_3^\ind(F)
\]
(see \cite{hutchinson-tao2009}). The indecomposable group $K_3^\ind(F)$ has been studied extensively in the literature (see for 
example \cite{merkurjev-suslin1990} or \cite{levine1989}). In many applications in algebraic $K$-theory and number theory it is important to 
understand the structure of the group $H_3(\SL_2(F),\z)$.  When $F$ is not quadratically closed, the above map has a  
nontrivial, and often quite large, kernel (see  \cite{hutchinson2013}, \cite{hutchinson2021}). 

In this paper we prove a refined version of the Bloch-Wigner exact sequence over domains of characteristic 2 (see Theorem \ref{char=2} 
for the general statement).  More precisely we show that  for any local domain of  characteristic 2,  where its residue field has more than 64 elements, 
we have an exact sequence 
\[
0 \arr \tors(\mu(A), \mu(A)) \arr H_3(\SL_2(A),\z) \arr \RB(A) \arr 0.
\]
This gives a positive answer to a question raised by Coronado and Hutchinson in \cite{C-H2022} over such rings and it improves 
similar results of Hutchinson (see \cite{hutchinson-2013},\cite{hutchinson2017}), as it leaves no ambiguity on 2-torsion elements.

The homology groups $H_\bullet(\SL_2(A),\z)$ are naturally $\RR_A:=\z[A^\times/(A^\times)^2]$-modules and this module structure 
plays a central role in the study of the homology groups $H_3(\SL_2(A),\z)$.

The refined Bloch group of a ring $A$, introduced by Hutchinson, is a certain subgroup of the the refined scissors congruence group 
of a $A$.  The refined scissors congruence group $\RP_1(A)$ of $A$ is defined by a presentation analogous to $\PP(A)$ but as a 
module over  the group ring $\RR_A$ rather than as an abelian group. 

 In a series of papers \cite{hutchinson2013}, \cite{hutchinson-2013}, \cite{hutchinson2017}, \cite{hutchinson-2017}, \cite{hutchinson2021},
 Hutchinson extensively studied the homology group $H_3(\SL_2(A),\z)$ when $A$ is a field or a local ring with sufficiently large residue field.
 Recently it has been proved that the third homology of $\SL_2$ over discrete valuation rings satisfies certain localisation property \cite{hmm2022}.
 
% and proved a refined version of the Bloch-Wigner exact sequence with coefficient in $\z\half$. Our results improves his results over local domains 
% of characteristic 2.

Our main result  follows from a carful analysis of a spectral sequence which converge to the homology of $\SL_2$. Our spectral sequence is a 
variant of a spectral sequence which is studied by Hutchinson in his series of papers and is similar to the one studies  for $\GL_2$ in 
\cite{mirzaii2011}. This variant has certain advantage when it comes to calculation of some differentials.

{\bf Notations.} In this note all rings are commutative, except probably group rings, and have the unit element $1$. For a ring $A$, let
$\WW_A$ be the set of $a\in \aa$ such that $1-a\in \aa$. Thus
\[
\WW_A:=\{a\in A: a(1-a)\in \aa\}.
\]
Let $\GG_A:=\aa/(\aa)^2$ and set $\RR_A:=\z[\GG_A]$. The element of $\GG_A$ represented by $a \in \aa$ is denoted by $\lan a \ran$.
We set $\Lan a\Ran:=\lan a\ran -1\in \RR_A$.

%%%%%%%%%%%%%%%%%%%%%%%%%%%%%%%%%%%%%%%%%%%%%%%%%%%%%%%%%%%%%%
\section{The scissors congruence groups}\label{sec1}
%%%%%%%%%%%%%%%%%%%%%%%%%%%%%%%%%%%%%%%%%%%%%%%%%%%%%%%%%%%%%%

Let $A$ be a ring. The {\it scissors congruence group} $\PP(A)$ of $A$ is defined as the quotient of the free abelian group 
generated by symbols $[a]$, $a\in \WW_A$, by the subgroup generated by elements
\[
[a] -[b]+\bigg[\frac{b}{a}\bigg]-\bigg[\frac{1- a^{-1}}{1- b^{-1}}\bigg]+ \bigg[\frac{1-a}{1-b}\bigg],
\]
where $a, b, a/b  \in \WW_A$. Let 
\[
S_\z^2(\aa):=(\aa \otimes \aa)/\lan a\otimes b+b\otimes a:a,b \in \aa\ran.
\]
The map
\[
\lambda: \PP(A) \arr S_\z^2(\aa), \ \ \ \ [a] \mapsto a \otimes (1-a),
\]
is a well-defined homomorphism and its kernel is called the {\it Bloch group} of $A$ (see \cite{dupont-sah1982}, \cite{suslin1991}, 
\cite{mirzaii2011})). We denote this kernel by $\BB(A)$. 

Let $\RP(A)$ be the quotient of the free $\RR_A$-module generated by symbols $[a]$, $a\in \WW_A$, by the $\RR_A$-submodule 
generated by the elements
\[
[a]-[b]+\lan a\ran\bigg[\frac{b}{a}\bigg]-\lan a^{-1}-1\ran\Bigg[\frac{1-a^{-1}}{1-b^{-1}}\Bigg] +\lan 1-a\ran\Bigg[\frac{1-a}{1-b}\Bigg],
\]
where $a,b,a/b \in \WW_A$. We have a natural surjective map $\RP(A)\arr \PP(A)$ and from the definition it follows immediately that
\[
\PP(A)\simeq \RP(A)_{\GG_A}=H_0(\GG_A, \RP(A)).
\]

Let $\II_A$ be the augmentation ideal of the group ring $\RR_A=\z[\GG_A]$. By direct computation one can show that the map
\[
\lambda_1: \RP(A) \arr \II_A^2, \ \ \ \ \ \ [a] \mapsto \Lan a \Ran\Lan 1-a\Ran
\]
is a well-defined $\RR_A$-homomorphism. If we consider $S_\z^2(\aa)$ as
a trivial $\GG_A$-module, then
\[
\lambda_2: \RP(A) \arr S_\z^2(\aa), \ \ \ \ \ [a] \mapsto a \otimes (1-a),
\]
is a homomorphism of $\RR_A$-modules. In fact $\lambda_2$ is
the composite $\RP(A) \arr \PP(A) \overset{\lambda}{\arr} S_\z^2(\aa)$.

Hutchinson defined the {\it refined scissors congruence group of $A$} as the $\RR_A$-module 
\[
\RP_1(A):=\ker(\lambda_1:\RP(A) \arr \II_A^2).
\]
He defined the {\it refined Bloch group} of $A$ as the $\RR_A$-module 
\[
\RB(A):=\ker(\lambda_2|_{\RP_1(A)}:\RP_1(A) \arr S_\z^2(\aa))
\] 
(see \cite[page 28]{hutchinson-2013}, \cite[Subsection 2.3]{hutchinson2013}).

%%%%%%%%%%%%%%%%%%%%%%%%%%%%%%%%%%%%%%%%%%%%%%%%%%%%%%%%%%%%%%
\section{The complex of unimodular vectors}
%%%%%%%%%%%%%%%%%%%%%%%%%%%%%%%%%%%%%%%%%%%%%%%%%%%%%%%%%%%%%%

The bridge between the third homology of $\SL_2(A)$ and the refined scissors congruence group usually is given by a complex 
built out of unimodular vectors.

Let $A$ be a commutative ring. A (column) vector ${\bf u}={\mtx {u_1} {u_2}}\in A^2$ is said to be unimodular if $u_1A+u_2A=A$.  
Equivalently, ${\bf u}={\mtx {u_1} {u_2}}$ is said to be unimodular if there exists a vector ${\bf v}={\mtx {v_1} {v_2}}$ such that 
$\{{\bf u}, {\bf v}\}$ is a basis of $A^2$.

For any non-negative integer $n$, let $X_n(A^2)$ be the free abelian group generated by the set of all $(n+1)$-tuples 
$(\lan{\bf v_0}\ran, \dots, \lan{\bf v_n}\ran)$, where every ${\bf v_i} \in A^2$ is unimodular and any two distinct vectors 
${\bf v_i}, {\bf v_j}$ are a basis of $A^2$. Observe that $\lan{\bf v}\ran\se A^2$ is the line generated by ${\bf v}\in A^2$, 
i.e. $\lan{\bf v}\ran={\bf v}A$.

We consider $X_n(A^2)$ as a left $\GL_2(A)$-module (respectively left $\SL_2(A)$-module) in a natural way. If necessary, we 
convert this action to a right action by the definition $m.g:=g^{-1}m$. Let us define the $n$-th differential operator
\[
\partial_n : X_n(A^2) \arr X_{n-1}(A^2), \ \ n\ge 1,
\]
as an alternating sum of face operators which throws away the $i$-th component of generators. Let $\partial_{-1}=\epsilon: 
X_0(A^2) \arr \z$ be defined by $\sum_i n_i(\lan {\bf v_{0,i}}\ran) \mt \sum_i n_i$. Hence we have the complex
\[
X_\bullet(A^2) \overset{\epsilon}{\arr} \z: \  \cdots \larr  X_2(A^2) \overset{\partial_2}{\larr} X_1(A^2) \overset{\partial_1}{\larr} 
X_0(A^2) \overset{\epsilon}{\arr}  \z \arr 0.
\]
The following is an important case that we should keep in our mind.

\begin{prp}[Hutchinson]\label{GE2C}
Let $A$ be a local ring. If $|A/\mmm_A|=k$, then the complex $X_\bullet(A^2) \arr \z$ is exact  in dimension $< k$.
\end{prp} 
\begin{proof}
See  \cite[Lemma~3.21]{hutchinson2017}.
\end{proof}

%%%%%%%%%%%%%%%%%%%%%%%%%%%%%%%%%%%%%%%%%%%%%%%%%%%%%%%%%%%%%%
\section{The main spectral sequence}
%%%%%%%%%%%%%%%%%%%%%%%%%%%%%%%%%%%%%%%%%%%%%%%%%%%%%%%%%%%%%%

Let $Z_i(A^2)=\ker(\partial_i)$. Assume that $X_\bullet(A^2) \arr \z$ is exact in dimension $<2$. Consider the complex
\begin{equation*}
0 \arr Z_1(A^2)  \overset{\inc}{\arr} X_1(A^2)  \overset{\partial_1}{\arr} X_0(A^2)  \arr 0.
\end{equation*}
Let $F_\bullet \arr \z$ be a projective resolution of $\z$ over $\SL_2(A)$. Let $D_{\bullet,\bullet}$ be the double complex
\[
D_{\bullet,\bullet}: 0 \arr F_\bullet \otimes_{\SL_2(A)} Z_1(A^2)  \overset{\id_{F_\bullet}\otimes\inc}{\larr} 
F_\bullet \otimes_{\SL_2(A)}X_1(A^2)  \overset{\id_{F_\bullet}\otimes\partial_1}{\larr} F_\bullet \otimes_{\SL_2(A)}X_0(A^2)  \arr 0.
\]
In our calculations  we usually use either the bar resolution $B_\bullet(\SL_2(A))\arr \z$ or the standard resolution
$C_\bullet(\SL_2(A))\arr \z$ \cite{brown1994}. Note that we have the natural morphisms 
\[
B_\bullet(\SL_2(A)) \arr C_\bullet(\SL_2(A)),  \ \ \ [h_1|h_2|\cdots|h_n]\mt (1,h_1,h_1h_2, \dots, h_1h_2\cdots h_n),
\]
\[
C_\bullet(\SL_2(A)) \arr B_\bullet(\SL_2(A)), \ \ \ (g_0,\dots, g_n) \mt  g_0[g_0^{-1}g_1| g_1^{-1}g_2|\cdots |g_{n-1}^{-1}g_n],
\]
which induce the identity on the homology of $\SL_2(A)$ with any coefficient \cite{brown1994}.

From the double complex $D_{\bullet,\bullet}$ we obtain the first quadrant spectral sequence
\[
E^1_{p.q}=\left\{\begin{array}{ll}
	H_q(\SL_2(A),X_p(A^2)) & p=0,1\\
	H_q(\SL_2(A),Z_1(A^2)) & p=2\\
	0 & p>2
\end{array}
\right.
\Longrightarrow H_{p+q}(\SL_2(A),\z).
\]
The group $\SL_2(A)$ acts transitively on the sets of generators of $X_i(A^2)$  for $i=0,1$. Let
\[
{\bf\infty}:=\lan {\bf e_1}\ran, \ \ \  {\bf 0}:=\lan {\bf e_2}\ran , \ \ \  {\bf a}:=\lan {\bf e_1}+ a{\bf e_2}\ran, \ \ \ a\in \aa.
\]
We choose $({\bf \infty})$ and  
$({\bf \infty} ,{\bf 0})$ as representatives of the orbit of the generators of $X_0(A^2)$ and $X_1(A^2)$  respectively.
Therefore 
\[
X_0\simeq \Ind _{B(A)}^{\SL_2(A)}\z, \ \ \ \ \ \ \ \ \  X_1\simeq \Ind _{T(A)}^{\SL_2(A)}\z,
\]
where 
\[
B(A):=\stabe_{\SL_2(A)}(\infty)=\Bigg\{\begin{pmatrix}
a & b\\
0 & a^{-1}
\end{pmatrix}:a\in \aa, b\in A\bigg\},
\]
\[
 T(A):=\stabe_{\SL_2(A)}(\infty,{\bf 0})=\Bigg\{\begin{pmatrix}
a & 0\\
0 & a^{-1}
\end{pmatrix}:a\in \aa\bigg\}.
\]
Note that $T(A)\simeq \aa$. In our calculations usually we identify $T(A)$ with $\aa$. By Shapiro's lemma we have
\[
E_{0,q}^1 \simeq H_q(B(A),\z), \  \  \ 
E_{1,q}^1 \simeq H_q(T(A),\z).
\]
In particular $E_{0,0}^1\simeq \z\simeq E_{1,0}^1$. Moreover
\[
d_{1, q}^1=H_q(\sigma) - H_q(\inc),
\]
where $\sigma: T(A) \arr B(A)$ is given by $\sigma(X)= wX w^{-1}$, $w={\mtxx 0 1 {-1} 0}$.

These implies that  $d_{1,0}^1$ is trivial, $d_{1,1}^1$  is induced by the map $T(A) \arr B(A)$ given by
$X\mt X^{-2}$ and $d_{1,2}^1$ is trivial.  Thus 
\[
\ker(d_{1,1}^1)\simeq\mu_2(A)=\{b\in \aa:b^2=1\}. 
\]
It is straightforward to check  that for any $b\in \mu_2(A)$,
\[
d_{2,1}^1\bigg([b]\otimes \partial_2(\infty, {\bf 0}, {\bf a})\bigg)=b. 
\]
Therefore $E_{1,1}^2=0$.

%%%%%%%%%%%%%%%%%%%%%%%%%%%%%%%%%%%%%%%%%%%%%%%%%%%%%%%%
\section{The refined scissors congruence group}
%%%%%%%%%%%%%%%%%%%%%%%%%%%%%%%%%%%%%%%%%%%%%%%%%%%%%%%%

Assume that the complex $X_\bullet(A^2) \arr \z$ is exact in dimension $<4$.  From the short exact sequence
\[
0 \arr  Z_2(A^2) \arr X_2(A^2) \arr Z_1(A^2) \arr 0
\]
we obtain the  long exact sequence
\begin{equation}\label{exatc1}
\!\!\!\!\!\!\!\!\!\!\!\!\!\!\!\!\!\!\!\!\!
%H_1(\SL_2(A), Z_2(A^2)) \arr 
H_1(\SL_2(A),X_2(A^2)) \arr H_1(\SL_2(A),Z_1(A^2)) \overset{\delta}{\arr} H_0(\SL_2(A),Z_2(A^2))   
\end{equation}
\[
\ \ \ \ \ \ \ \ \ \ \ \ \ \ \ \ \  \ \ \ \ \ \ \  \ \ \ \ \ \ \  \ \ \ \ \ \ \ 
\arr H_0(\SL_2(A),X_2(A^2)) \arr H_0(\SL_2(A),Z_1(A^2)) \arr 0.
\]
Choose $(\infty, {\bf 0}, {\bf a})$, $\lan a\ran \in \GG_A$, as representatives of the orbits of the generators of $X_2(A^2)$. Then
\[
X_2\simeq \bigoplus_{\lan a\ran\in \GG_A} \Ind _{\mu_2(A)}^{\SL_2(A)}\z\lan a\ran, 
\]
where $\mu_2(A)=\stabe_{\SL_2(A)}(\infty,{\bf 0}, {\bf a})$.  Thus
\[
H_q(\SL_2(A),X_2(A^2)) \simeq \bigoplus_{\lan a\ran\in \GG_A} H_q(\mu_2(A), \z)\simeq \RR_A\otimes H_q(\mu_2(A),\z).
\]
From the exact sequence $X_4(A^2) \overset{\partial_4}{\arr} X_3(A^2)  \overset{\partial_3}{\arr} Z_2(A^2) \arr 0$
 we obtain the exact sequence
 \[
 X_4(A^2)_{\SL_2(A)} \overset{\overline{\partial_4}}{\arr} X_3(A^2)_{\SL_2(A)}    \overset{\overline{\partial_3}}{\arr} Z_2(A^2)_{\SL_2(A)}  \arr 0
 \]
 of $\RR_A$-modules. The orbits of the action of  $\SL_2(A)$ on $X_3(A)$  and $X_4(A)$ are  represented by
\[
\lan a\ran[x]:=(\lan {\bf e_1}\ran,\lan{\bf e_2}\ran, \lan{\bf e_1}+a {\bf e_2} \ran, \lan{\bf e_1}+ax{\bf e_2}\ran),  \  \  \  \   
\lan a\ran\in \GG_A, x\in \WW_A,
\]
and 
\[
\lan a\ran[x,y]:=(\lan{\bf e_1}\ran, \lan{\bf e_2}\ran, \lan{\bf e_1}+a{\bf e_2}\ran, 
\lan{\bf e_1}+ax{\bf e_2}\ran, \lan{\bf e_1}+ay{\bf e_2}\ran),  \ \  \lan a\ran\in \GG_A, x,y,x/y\in \WW_A,
\]
respectively. Thus $X_3(A^2)_{\SL_2(A)} $ is the free $\RR_A$-module generated by the symbols
$[x]$, $x\in \WW_A$ and  $X_4(A^2)_{\SL_2(A)} $ is the free $\RR_A$-module generated by the symbols
$[x,y]$, $x,y,x/y\in \WW_A$. It is straightforward to check that 
\[
\overline{\partial_4}([x,y])=[x]-[y]+\lan x\ran\bigg[\frac{y}{x}\bigg]- \lan x^{-1}-1\ran\Bigg[\frac{1-x^{-1}}{1-y^{-1}}\Bigg] - 
\lan 1-x\ran\Bigg[\frac{1-x}{1-y}\Bigg].
\]
These imply that
\[
H_0(\SL_2(A),Z_2(A^2)) \simeq \RP(A).
\]
It is not difficult to see that the composite
\[
 \RP(A)\simeq H_0(\SL_2(A),Z_2(A^2)) \arr  H_0(\SL_2(A),X_2(A^2))\simeq \z[\GG_A]=\RR_A
\]
is given by
\[
[x]\mapsto \Lan x\Ran\Lan 1-x\Ran.
\]
This is the map $\lambda_1$ discussed in  Section \ref{sec1}. 

%From the exact sequence $X_3(A^2) \overset{\partial_3}{\arr} X_2(A^2)  \overset{\partial_2}{\arr} Z_1(A^2) \arr 0$ we obtain the exact sequence
% \[
% X_3(A^2)_{\SL_2(A)} \overset{\overline{\partial_3}}{\arr} X_2(A^2)_{\SL_2(A)}    \overset{\overline{\partial_2}}{\arr} Z_1(A^2)_{\SL_2(A)}  \arr 0
% \]
% of $\RR_A$-modules. This induces the isomorphism
% \[
% H_0(\SL_2(A),Z_1(A^2))\simeq \RR_A/\bigg\lan \Lan a\Ran \Lan 1-a\Ran: a, 1-a\in \aa\bigg\ran.
% \]
% We denote this group with $\GW(A)$ and call it the Grothendieck-Witt ring of $A$. The following ideal of $\GW(A)$,
% \[
% I(A):=\II_A/\bigg\lan \Lan a\Ran \Lan 1-a\Ran: a, 1-a\in \aa\bigg\ran
% \]
% is called the {\it fundamental ideal} of $A$. Thus from the exact sequence (\ref{exatc1}) 
 %finds the following form
%\[
%H_1(\SL_2(A), Z_2(A^2)) \arr \RR_A\otimes\mu_2(A) \arr E_{2,1}^1 \arr \mathcal{RP}(A) 
%\arr \RR_A \arr \GW(A) \arr 0.
%\]
we obtain the exact sequence
\begin{equation}\label{exact2}
H_1(\SL_2(A), Z_2(A^2)) \arr \RR_A\otimes\mu_2(A) \arr E_{2,1}^1 \arr \mathcal{RP}_1(A) \arr 0 .
\end{equation}

%%%%%%%%%%%%%%%%%%%%%%%%%%%%%%%%%%%%%%%%%%%
\section{The refined bloch group}
%%%%%%%%%%%%%%%%%%%%%%%%%%%%%%%%%%%%%%%%%&&&

Let $A$ be a ring such that $\mu_2(A)=1$ (e.g. a domains of characteristic two). Moreover let
$X_\bullet(A^2) \arr \z$ is exact in dimension $<4$.  It follows from the exact sequence (\ref{exact2}) that
\[
E_{2,1}^1=H_1(\SL_2(A),Z_1(A^2)) \simeq \mathcal{RP}_1(A).
\]
Since $\ker(d_{1,1})\simeq \mu_2(A)=1$, we have $E_{2,1}^2=E_{2,1}^1$.

The group $B(A)$ sits in the extension 
\[
1\arr  N(A) \arr B(A) \overset{\alpha}{\arr}  T(A) \arr 1, 
\]
where $N(A):=\Bigg\{\begin{pmatrix}
1 & b\\
0 & 1
\end{pmatrix} \in \SL_2(A)| b\in A \Bigg \}$  and $\alpha \begin{pmatrix}
a & b\\
0 & a^{-1}
\end{pmatrix}=\begin{pmatrix}
a & 0\\
0 & a^{-1}
\end{pmatrix}$.
Note that $N(A)\simeq A$.  This extension splits canonically. Let
 \[
H_2(B(A),\z) \simeq H_2(T(A),\z) \oplus \SS_2\simeq (\aa \wedge \aa) \oplus \SS_2
\]
and consider the differential 
\[
d^2_{2,1}: \RP_1(A)\simeq E^2_{2,1}\to H_2(B(A),\z)\simeq (A^{\times}\wedge A^{\times})\oplus \SS_2.
\]

\begin{lem}\label{ker=RB}
The kernel of the composite $\RP_1(A) \overset{d_{2,1}^ 2}{\larr} H_2(B(A),\z) \overset{\alpha_\ast}{\larr}  A^{\times}\wedge A^{\times} $
is isomorphic with the refined Bloch group $\RB(A)$.
\end{lem}
\begin{proof}
From the complex $0 \arr Z_1(A^2)  \overset{\inc}{\arr} X_1(A^2)  \overset{\partial_1}{\arr} X_0(A^2)  \arr 0$ we obtain the
 first quadrant spectral sequence
\[
\EE^1_{p.q}=\left\{\begin{array}{ll}
	H_q(\GL_2(A),X_p(A^2)) & p=0,1\\
	H_q(\GL_2(A),Z_1(A^2)) & p=2\\
	0 & p>2
\end{array}
\right.
\Longrightarrow H_{p+q}(\GL_2(A),\z).
\]
This spectral sequence have been studied in \cite{mirzaii2011}. The inclusion $\SL_2(A) \harr \GL_2(A)$ gives us the morphism of spectral sequences
\[
\begin{tikzcd}
	E^1_{p,q} \ar[d] \ar[r,Rightarrow] & H_{p+q}(\SL_2(A),\z)\ar[d]\\
	\EE^1_{p,q} \ar[r,Rightarrow] & H_{p+q}(\GL_2(A),\z).
\end{tikzcd}
\]
From this we obtains the commutative diagram
\[
\begin{tikzcd}
\RP_1(A) \ar[d] \ar[r, "\alpha_\ast\circ d_{2,1}^2 "] \ar[d]& 
%H_2(B(A),\z) \ar{r}\ar[d]  &
H_2(T(A),\z)\ar[r, "\simeq"] \ar[d]&\aa\wedge \aa \ar[d, hook]\\
\PP(A) \ar[r,"\beta_\ast\circ d_{2,1}^2 "] & 
%\displaystyle\frac{H_2(B_2(A),\z)}{(\sigma_\ast-\inc_\ast)(H_2(T_2(A),\z))} \ar[r]&
\displaystyle\frac{H_2(T_2(A),\z)}{(\sigma_\ast-\inc_\ast)(H_2(T_2(A),\z))} \ar[r, "\simeq"] & (\aa\wedge\aa) \oplus S_\z^2(A),
\end{tikzcd}
\]
where
\[
B_2(A):=\stabe_{\GL_2(A)}(\infty)=\Bigg\{\begin{pmatrix}
a & b\\
0 & d
\end{pmatrix}:a, d\in \aa, b\in A\bigg\},
\]
\[
 T_2(A):=\stabe_{\GL_2(A)}(\infty,{\bf 0})=\Bigg\{\begin{pmatrix}
a & 0\\
0 & d
\end{pmatrix}:a, d\in \aa\bigg\},
\]
and $\beta: B_2(A) \arr T_2(A)$ is defined similar to $\alpha$. It is straightforward to check that
the vertical map on the right is given by
\[
 a\wedge b\arr \Big(2(a\wedge b), 2(a\otimes b)\Big).
\]
Moreover by a  tedious computation one can show that the bottom horizontal composite is given by
\[
[a]\mt (a\wedge (1-a), -a \otimes (1-a)),
\]
(for details see \cite{mirzaii2017}, \cite{mirzaii2011}). The lemma follows from the commutativity of the above diagram.
\end{proof}

Our calculations suggest that the image of the differential
\[
d^2_{2,1}: \RP_1(A)\simeq E^2_{2,1}\to H_2(B(A),\z)\simeq (A^{\times}\wedge A^{\times})\oplus \SS_2.
\]
may have a nonzero part in $\SS_2\se H_2(B(A),\z)$ (see \cite{mirzaii2017}). To obtain a good  relation between the third homology of  $\SL_2$ and  
the refined Bloch group $\RB(A)$,  we consider those rings that the summand $\SS_n$ of $H_n(B(A),\z)$ is trivial when
$n=2,3$. The isomorphism $H_n(T(A),\z)\overset{\inc_\ast}{\simeq} H_n(B(A),\z)$ has been already studied  for a certain class 
of rings. For example we have the following result of Hutchinson.

\begin{prp}\label{iso-hut}
Let $A$ be any local integral domain. If the residue field $A/\mmm_A$ is finite of order $p^d$ we suppose that $(p -1)d > 2n$.
Then the natural maps $T(A) \arr B(A)$ induces the isomorphism $H_n(T(A), \z)\simeq H_n(B(A), \z)$.
\end{prp}
\begin{proof}
See \cite[Proposition 3.19]{hutchinson2017}.
\end{proof}

One more fact that is needed in the next section is the structure of the homology group $H_3(T(A),\z)$.

Let $B$ be an abelian group. Let $\sigma_1:\tors(B,B)\arr \tors(B,B)$
be obtained by interchanging  the group $B$. It is not difficult to show that $\sigma_1$ is induced by
the involution $B\otimes B \arr B\otimes B$, $a\otimes b\mapsto -b\otimes a$.

Let $\Sigma_2'=\{1, \sigma'\}$ be the symmetric group of order 2. Consider the following action of $\Sigma_2'$ on $\tors(B,B)$:
\[
(\sigma', x)\mapsto -\sigma_1(x).
\]

\begin{prp}\label{H3B}
For any abelian group $B$ we have the exact sequence
\[
\begin{array}{c}
0 \arr \bigwedge_\z^3 B \arr H_3(B,\z) \arr \tors(B,B)^{\Sigma_2'} \arr 0,
\end{array}
\]
where the right side homomorphism  is obtained from the composition
\[
H_3(B,\z) \overset{{\Delta_B}_\ast}{\larr } H_3(B\oplus B,\z) \arr \tors(B,B),
\]
$\Delta_B$ being the diagonal map $B \arr B\oplus B$, $b \mt (b,b)$.
\end{prp}
\begin{proof}
See \cite[Lemma~5.5]{suslin1991}, \cite[Section 6]{breen1999}.
\end{proof}

%%%%%%%%%%%%%%%%%%%%%%%%%%%%%%%%%%%%%%%%%%%
\section{The refined bloch-wigner exact sequence in characteristic 2}
%%%%%%%%%%%%%%%%%%%%%%%%%%%%%%%%%%%%%%%%%&&&

Let $A$ be a ring which satisfies the following  three conditions:
\bigskip

\par\ \ \  (1) $\mu_2(A)=1$,
\par \ \ \ (2) $X_\bullet(A^2) \arr \z$ is exact in dimension $<4$,
\par \ \ \ (3) $H_n(T(A),\z)\simeq H_n(B(A),\z)$ for $n=2,3$.
\bigskip

Note that if $\mu_2(A)=1$, then $A$ is of characteristic $2$. The main example that we should keep in our mind is a
domain of characteristic $2$ such that its residue field has more than 64 elements (see Propositions \ref{GE2C} and \ref{iso-hut}).
%\begin{exa}
%Note that $\mu_2(A)=1$ implies that $A$ is of characteristic $2$. The following rings satisfy in the above conditions:
%\par (i) If $A$ is a $\F_2$-algebra with many units such that $\mu_2(A)=1$.
%\par (ii) If $A$ is a local $\F_2$-algebra such that its residue field has more than ? elements and $\mu_2(A)=1$,
%\par (iii)  If $A$ is a local domain of characteristic 2 such that its residue field has more than 64 elements.
%\end{exa}
%
Here is our main result.

\begin{thm}\label{char=2}
Let $A$ be a ring which satisfies in the conditions {\rm (1)},  {\rm (2)}  and {\rm (3)}.  
Then we have the refined Bloch-Wigner exact sequence
\[
\tors(\mu(A),\mu(A))^{\Sigma_2'} \arr H_3(\SL_2(A),\z) \arr \RB(A) \arr 0.
\]
Moreover if $A$ is a domain, then we have the exact sequence
\[
0 \arr \tors(\mu(A),\mu(A)) \arr H_3(\SL_2(A),\z) \arr \RB(A) \arr 0.
\]
\end{thm}
\begin{proof}
By Lemma \ref{ker=RB}, $E^\infty_{2,1} \simeq E^3_{2,1} \simeq \RB(A)$.  We show that the differential
\[
d_{2,2}^1: H_2(\SL_2(A),Z_1(A^2))\arr \aa\wedge\aa
\]
is surjective. For $a\in A^{\times}$, denote  $(\infty, {\bf 0}, {\bf a})\in X_2(A^2)$ 
%and $({\bf 0},\infty, {\bf a})$ of $X_2(A^2)$  
by $X_a$. 
%and $X'_a$, respectively. 
Let $Y=(\infty,{\bf 0})+({\bf 0},\infty)\in Z_1(A^2)$. For $a,b\in \aa$, let 
\[
\lambda(a,b)\in H_2(\SL_2(A),Z_1(A))=H_2(B_\bullet(\SL_2(A))\otimes_{\SL_2(A)} Z_1(A))
\]
 be the element
\begin{align*}
\lambda(a,b):=&([a|b]+[w|ab]-[w|a]-[w|b])\otimes Y+[wab|wab]\otimes \partial_2(X_{ab})\\
&-[wa|wa]\otimes \partial_2(X_a) - [wb|wb]\otimes \partial_2(X_b) + [w|w]\otimes \partial_2(X_1).
\end{align*}
Recall that $w=\mtxx{0}{1}{-1}{0}=\mtxx{0}{1}{1}{0}$. We have	
\begin{align*}
d^1_{2,2}(\lambda(a,b))=&(w+1)([a|b]+[w|ab]-[w|a]-[w|b])\otimes (\infty,{\bf 0})\\
&+(g_{ab}^{-1}-h_{ab}^{-1}+1)([wab|wab])\otimes (\infty,{\bf 0})\\
&-(g_{a}^{-1}-h_{a}^{-1}+1)([wa|wa])\otimes (\infty, {\bf 0})\\
&-(g_{b}^{-1}-h_{b}^{-1}+1)([wb|wb])\otimes (\infty, {\bf 0})\\
&+(g_1^{-1}-h_1^{-1}+1)([w|w])\otimes (\infty, {\bf 0}),
\end{align*}
where 
$g_x:=\begin{pmatrix}
0 & 1\\
1 & x
\end{pmatrix}$ and 
$h_x:=\begin{pmatrix}
1 & x^{-1}\\
0 & 1
\end{pmatrix}$. 
This element is in $H_2(\SL_2(A),X_1(A^2))=H_2(B_\bullet(\SL_2(A))\otimes_{\SL_2(A)} X_1(A))$. The morphisms
\[
B_\bullet(\SL_2(A))\otimes_{\SL_2(A)} X_1(A) \arr B_\bullet(\SL_2(A))\otimes_{T(A)} \z \arr C_\bullet(\SL_2(A))\otimes_{T(A)} \z,
\]
\[
[g_1|\cdots|g_n]\otimes (\infty, {\bf 0}) \mt [g_1|\cdots|g_n]\otimes 1 \mt \otimes (1, g_1, \dots, g_1\cdots g_n)\otimes 1,
\]
induce the isomorphisms
\[
H_2(B_\bullet(\SL_2(A))\otimes_{\SL_2(A)} X_1(A)) \simeq H_2(B_\bullet(\SL_2(A))\otimes_{T(A)} \z)\simeq  H_2(C_\bullet(\SL_2(A))\otimes_{T(A)} \z).
\]
Following these maps we see that $d^1_{2,2}(\lambda(a,b))$ as an element of $H_2(C_\bullet(\SL_2(A))\otimes_{T(A)} \z)$ find the following form
\begin{align*}
d^1_{2,2}(\lambda(a,b))=&\bigg((w,wa,wab)+(w,1,ab)-(w,1,a)-(w,1,b)\\
&+(1,a,ab)+(1,w,wab)+(1,w,wa)+(1,w,wb))\otimes 1\\
&+((g^{-1}_{ab},g^{-1}_{ab}wab,g^{-1}_{ab})-(h^{-1}_{ab},h^{-1}_{ab}wab,h^{-1}_{ab})+(1,wab,1))\otimes 1\\
&-((g^{-1}_{a},g^{-1}_{a}wa,g^{-1}_{a})-(h^{-1}_{a},h^{-1}_{a}wa,h^{-1}_{a})+(1,wa,1))\otimes 1\\
&-((g^{-1}_{b},g^{-1}_{b}wb,g^{-1}_{b}))-(h^{-1}_{b},h^{-1}_{b}wb,h^{-1}_{b})+(1,wb,1)\otimes 1\\
&+((g^{-1}_1,g^{-1}_1w,g^{-1}_1)-(h^{-1}_1,h^{-1}_1w,h^{-1}_1)+(1,w,1)\bigg)\otimes 1.
\end{align*}
Now we want to find a representative of this element in $C_\bullet(T(A))\otimes_{T(A)}\z$ by the isomorphism
\[
H_2(C_\bullet(\SL_2(A))\otimes_{T(A)} \z))\simeq H_2(C_\bullet(\aa))\otimes_{\aa} \z).
\]
Let  $s:\SL_2(A)\backslash T(A)\to \SL_2(A)$  be any (set-theoretic) section of the canonical projection 
$\pi:\SL_2(A)\to \SL_2(A)\backslash T(A)$. For $g\in \SL_2(A)$, set $\overline{g}=(g)(s\circ \pi(g))^{-1}$. Then the morphism
\[
C_\bullet(\SL_2(A))\otimes_{T(A)} \z \overset{s_\bullet}{\larr} C_\bullet(T(A))\otimes_{T(A)}\z, \ \ \ 
[g_1|\dots|g_n]\otimes 1 \mt [\overline{g_1}|\dots|\overline{g_n}]\otimes 1
\]
induces the desired isomorphism. By choosing the section
\[
s(T(A)\begin{pmatrix}
a&b\\
c&d
\end{pmatrix})
=\begin{cases}	
\begin{pmatrix}
1&a^{-1}b\\
ac&ad
\end{pmatrix} & \text{if  $a\neq 0$}\\
\\
\begin{pmatrix}
0&1\\
1&bd
\end{pmatrix} &\text{if $a=0$}
\end{cases}
\]
We see that $d^1_{2,2}(\lambda(a,b))$ in $H_2(C_\bullet(T(A))\otimes_{T(A)}\z)=H_2(C_\bullet(\aa)\otimes_{\aa}\z)$ is of the following form
 \begin{align*}
d^1_{2,2}(\lambda(a,b))=&\bigg((1,a^{-1},(ab)^{-1})+(1,a,ab)+(1,(ab)^{-1},1)-(1,a^{-1},1)-(1,b^{-1},1)\bigg)\otimes 1.
\end{align*}
In $H_2(B_{\bullet}(\aa)\otimes_{\aa} \z)$ this elements corresponds to
\begin{align*}
d^1_{2,2}(\lambda(a,b))=&\bigg([a^{-1}|b^{-1}] +[a|b]+[a^{-1}b^{-1}|ab]-[a^{-1}|a]-[b^{-1}|b]\bigg)\otimes 1.
\end{align*}
Now by adding the following null element in $H_2(A^{\times},\z)=H_2(B_{\bullet}(\aa)\otimes_{\aa} \z)$
\[
d_3([a^{-1}|b^{-1}|ab]-[b^{-1}|b|a])\otimes 1,
\]
it is easy to see that
\[
d^1_{2,2}(\lambda(a,b))=\Big([a|b]-[b|a]\Big)\otimes 1 \in H_2(B_{\bullet}(\aa)\otimes_{\aa} \z).
\]
This shows that $d^1_{2,2}$ is surjective. Therefore $E_{1,2}^2=0$.

Now we need to study $E_{0,3}^\infty=E_{0,3}^3$. To do this, first consider the differential 
\[
d_{1,3}^1=H_3(\sigma)-H_3(\inc): H_3(T(A),\z) \to H_3(B(A),\z)\simeq H_3(T(A),\z) . 
\]
By Proposition \ref{H3B}, we have the exact sequence
\[
\begin{array}{c}
0 \arr \bigwedge_\z^3 \aa \arr H_3(T(A),\z) \arr \tors (\mu(A), \mu(A))^{\Sigma_2'} \arr 0.
\end{array}
\]
%which splits (non-canonically) \cite{suslin1991}.
It is straightforward to check  that $d_{1,3}^1\mid_{\bigwedge_\z^ 3 \aa}$ coincides with multiplication by $2$. 
%Moreover $d_{1,3}^1$ induces the zero map on $\tors (\mu(A), \mu(A))^{-\sigma}$!!!!!????. 
Thus we have the exact sequence
\[
\begin{array}{c}
\bigwedge_\z^3 \aa/2 \arr E_{0,3}^2 \arr \tors(\mu(A), \mu(A))^{\Sigma_2'}/\TT \arr 0,
\end{array}
\]
for some  subgroup $\TT$ of $\tors(\mu(A), \mu(A))^{\Sigma_2'}$.
By an easy analysis of the spectral sequence we have the exact sequence
\[
E_{0,3}^3 \arr H_3(\SL_2(A),\z) \arr \RB(A) \arr 0.
\]
We denote the image of $a\wedge b\wedge c \in \bigwedge_\z^3 \aa/2$ in $E_{0,3}^ 2$ again by $a\wedge b\wedge c$.
%Observe that the composite
%\[
%\begin{array}{c}
%\bigwedge_\z^3 \aa/2 \arr E_{0,3}^ 2 \arr H_3(\SL_2(A),\z)
%\end{array}
%\]
%is given by
%\[
%a\wedge b\wedge c \mapsto {\bf c}\bigg( {\mtxx a 0 0 {a^{-1}}}, {\mtxx b 0 0 {b^{-1}}}, {\mtxx c 0 0 {c^{-1}}} \bigg).
%\]
Since $d_{2,2}^1(\lambda(ab,c)-\lambda(a,c)-\lambda(b,c))=0$, we have $\lambda(ab,c)-\lambda(a,c)-\lambda(b,c) \in E_{2,2}^2$.
We show that 
\begin{equation}\label{d3}
d_{2,2}^2(\lambda(ab,c)-\lambda(a,c)-\lambda(b,c))= -a\wedge b\wedge c
%{\bf c}\bigg( {\mtxx a 0 0 {a^{-1}}}, {\mtxx b 0 0 {b^{-1}}}, {\mtxx c 0 0 {c^{-1}}} \bigg) 
\in E_{0,3}^2.
%H_3(T(A),\z)/\im(H_3(\sigma)-H_3(\inc))
\end{equation}
This would imply that there is a surjective map $ \tors (\mu(A), \mu(A))^{\Sigma_2'} \two E_{0,3}^3$ and therefore we obtain the exact sequence
\[
 \tors (\mu(A), \mu(A))^{\Sigma_2'}  \arr H_3(\SL_2(A),\z) \arr \RB(A) \arr 0,
\]
which proof the first claim of the theorem.
Now we prove the equality (\ref{d3}). 

Consider the diagram
\[
\begin{tikzcd}
B_3(\SL_2(A))\otimes
%_{\SL_2(A)} 
X_0(A^ 2) & B_3(\SL_2(A))\otimes
%_{\SL_2(A)} 
X_1(A^ 2) \ar["\id_{B_3}\otimes \partial_1"',  l] \ar[d, "d_3\otimes \id_{X_1}"] & \\
& B_2(\SL_2(A))\otimes
%_{\SL_2(A)} 
X_0(A^ 2) & B_2(\SL_2(A))\otimes  Z_1(A^ 2).\ar[l, "\id_{B_2}\otimes \inc"']
\end{tikzcd}
\]
The element $\Lambda(a,b,c):=\lambda(ab,c)-\lambda(a,c)-\lambda(b,c)$ is
\begin{align*}
\Lambda(a,b, c)&\!:=\!\Big([ab|c]\!-\![a|c]\!-\![b|c]\!+\![w|abc]\!-\![w|ac]\!-\![w|bc]\!-\![w|ab]\!+\![w|a]\!+\![w|b]\!+[w|c]\Big)\!\otimes\! Y\\
&+[wabc|wabc]\otimes \partial_2(X_{abc})-[wab|wab]\otimes \partial_2(X_{ab})-[wbc|wbc]\otimes \partial_2(X_{bc})\\
&-[wac|wac]\otimes \partial_2(X_{ac})+[wa|wa]\otimes \partial_2(X_a) + [wb|wb]\otimes \partial_2(X_b) \\
&+ [wc|wc]\otimes \partial_2(X_c) - [w|w]\otimes \partial_2(X_1).
\end{align*}
For an element $z\in \aa$, consider $[wz|wz]\otimes \partial_2(X_z)\in B_2(\SL_2(A))\otimes_{\SL_2(A)} Z_1(A^ 2)$. 
For the matrices $g_z$ and $h_z$ we have the identities
\[
\begin{array}{cccc}
z^{-1}g_{z^{-1}}=g_zz, & zh_z=h_{z^{-1}}z, & g^{-1}_z w=h_{z^{-1}}, & h^{-1}_z=h_z.
\end{array}
\]
Using these identities we obtain
\begin{align*}
[wz|wz]\otimes \partial_2(X_z)&\!=\![wz|wz]\!\otimes\!(\infty,0)\\
&\!\!\!\!\!\!\!\!\!+\!(d_3\!\otimes\!\id_{X_1}\!)(([g^{-1}_z|wz|wz]\!-\![h_z|wz|wz]\!+\![z^{-1}|g^{-1}_z|wz]\!-\![z|h_z|wz])\!\otimes\!(\infty,{\bf 0}))\\
&\!\!\!\!\!\!\!\!\!+(d_3\otimes\id_{X_1})(([z|z^{-1}|g_z]-[z^{-1}|z|h_z]+[z^{-1}|z|z^{-1}])\otimes(\infty,0)).
\end{align*}
If
\begin{align*}
\theta_z&:=\![g^{-1}_z|wz|wz]\!-\![h_z|wz|wz]\!+\![z^{-1}|g^{-1}_z|wz]\!-\![z|h_z|wz]\\
&+\![z|z^{-1}|g_z]\!-\![z^{-1}|z|h_z]\!+\![z^{-1}|z|z^{-1}],
\end{align*}
then by a direct calculation we have
\[
\Lambda(a,b,c)=(d_3\otimes\id_{X_1})(([w|ab|c]-[w|a|c]-[w|b|c])\otimes Y+(\Phi_{a,b, c}+\Psi_{a,b, c})\otimes (\infty,{\bf 0})),
\]
where
\begin{align*}
\Phi_{a,b, c}&=\theta_{abc}-\theta_{ab}-\theta_{bc}-\theta_{ac}+\theta_{a}+\theta_{b}+\theta_{c}-\theta_{1}\\
\Psi_{a,b,c}&=[wab|wab|c]-[wa|wa|c]-[wb|wb|c]+[w|w|c]\\
&+[c|wabc|wabc]-[c|wac|wac]-[c|wbc|wbc]+[c|wc|wc]\\
&+[ab|wab|c]-[a|wa|c]-[a|c|wac]+[ab|c|wabc]-[c|ab|wabc]+[c|a|wac]\\
&-[b|wb|c]-[b|c|wbc]+[c|b|wbc]-[b|a|c]+[b|c|a]-[c|b|a].
\end{align*}
Since $\partial_1(Y)=0$, we need only to study
\[
(\id_{B_3}\otimes \partial_1)((\Phi_{a,b, c}+\Psi_{a,b, c})\otimes (\infty,{\bf 0})).
\]
Through the maps
\[
B_3(\SL_2(A))\otimes_{\SL_2(A)} X_1(A^2) \overset{\id_{B_3}\otimes \partial_1}{\larr} B_3(\SL_2(A))\otimes_{\SL_2(A)} X_0(A^2) 
\]
the above element maps to
\[
(w\Phi_{a,b, c}-\Phi_{a,b, c})\otimes (\infty)+ (w\Psi_{a,b, c}-\Psi_{a,b, c})\otimes (\infty).
\]
Now consider the composite
\[
B_3(\SL_2(A))\otimes_{\SL_2(A)} X_0(A^2) \arr C_3(\SL_2(A))\otimes_{\SL_2(A)} X_0(A^2) \arr C_3(\SL_2(A))\otimes_{B(T)} \z.
\]
Then
\begin{align*}
(w\theta_z-\theta_z)\otimes (\infty)\mt&\bigg\{(w,wg^{-1}_z,wg^{-1}_zwz,wg^{-1}_z)-(w,wh_z,wh_zwz,wh_z)\\
&+(w,wz^{-1},wz^{-1}g^{-1}_zwz)-(w,wz,wzh_z,wzh_zwz)\\
&+(w,wz,w,wg^{-1}_z)-(w,wz^{-1},w,wh_z)\\
&+(w,wz^{-1},w,wz^{-1})-(1,g^{-1}_z,g^{-1}_zwz,g^{-1}_z)\\
&+(1,h_z,h_zwz,h_z)-(1,z^{-1},z^{-1}g^{-1}_z,z^{-1}g^{-1}_zwz)\\
&+(1,z,zh_z,zh_zwz)-(1,z,1,g^{-1}_z)\\
&+(1,z^{-1},1,h_z)-(1,z^{-1},1,z^{-1})\bigg\}\otimes 1
\end{align*}
and
\begin{align*}
(w\Psi_{a,b, c}-\Psi_{a,b, c})\otimes (\infty)\mt &\bigg\{(w,ab,w,wc)-(w,a,w,wc)-(w,b,w,wc)+(w,1,w,wc)\\
&+(w,wc,ab,wc)-(w,wc,a,wc)-(w,wc,b,wc)+(w,wc,1,wc)\\
&+(w,wab,1,c)-(w,wa,1,c)-(w,wa,wac,1)+(w,wab,wabc,1)\\
&-(w,wc,wabc,1)+(w,wc,wac,1)-(w,wb,1,c)-(w,wb,wbc,1)\\
&+(w,wc,wbc,1)-(w,wb,wab,wabc)+(w,wb,wbc,wabc)\\
&-(w,wc,wbc,wabc)-(1,wab,1,c)+(1,wa,1,c)+(1,wb,1,c)\\
&-(1,w,1,c)-(1,c,wab,c)+(1,c,wa,c)+(1,c,wb,c)-(1,c,w,c)\\
&-(1,ab,w,wc)+(1,a,w,wc)+(1,a,ac,w)-(1,ab,abc,w)\\
&+(1,c,abc,w)-(1,c,ac,w)+(1,b,w,wc)+(1,b,bc,w)\\
&-(1,c,bc,w)+(1,b,ab,abc)-(1,b,bc,abc)+(1,c,bc,abc)\bigg\}\otimes 1.
\end{align*}
Now we want to follow these elements through  the maps
\[
C_3(\SL_2(A))\otimes_{B(A)} \z \overset{s}{\larr}  C_3(B(A))\otimes_{B(A)} 
\z \arr C_3(T(A))\otimes_{T(A)} \z \arr B_3(T(A))\otimes_{T(A)} \z
\]
where $s:\SL_2(A)\backslash T(A)\to \SL_2(A)$ is the section discussed in above.
%We have the equalities of matrices
%\[
%\begin{array}{llll}
%wg^{-1}_z=\mtxx{1}{0}{z}{1}, & wg^{-1}_zwz=z^{-1}\mtxx{0}{1}{1}{z^{-1}}, & wh_z=\mtxx{0}{1}{1}{z^{-1}}, & wh_zwz=z\mtxx{1}{0}{z}{1}\\
%wz^{-1}g^{-1}_z=z\mtxx{1}{0}{z}{1}, & wz^{-1}g^{-1}_zwz=\mtxx{0}{1}{1}{z}, & wzh_z=z^{-1}\mtxx{0}{1}{1}{z^{-1}}, & wzh_zwz=\mtxx{1}{0}{z}{1}\\
%wg^{-1}_z=\mtxx{1}{0}{z}{1}, & g^{-1}_zwz=z\mtxx{1}{z^{-1}}{0}{1}, & h_zwz=\mtxx{1}{z^{-1}}{z}{0}, & z^{-1}g^{-1}_z=\mtxx{1}{z^{-1}}{z}{0}\\
%z^{-1}g^{-1}_zwz=\mtxx{1}{z^{-1}}{0}{1}, & zh_z = z\mtxx{1}{z^{-1}}{0}{1}, & zh_zwz=z\mtxx{1}{z^{-1}}{z}{0}, &g^{-1}_z=z\mtxx{1}{z^{-1}}{z}{0}
%\end{array}
%\]
It is straightforward to check that modulo $\text{im}(d_4)$ we have
\[
(w\theta_z-\theta_z)\otimes(\infty)\mt -[z^{-1}|z|z^{-1}]\otimes 1=[z|z^{-1}|z]\otimes 1
\]	
%Let $R_z:=[z|z^{-1}|z]\otimes 1$ (note that $R_1$ is null). 
Moreover
\begin{align*}
(w\Psi_{a,b, c}-\Psi_{a,b, c})\otimes (\infty)\mt &\bigg\{[c^{-1}|abc|(abc)^{-1}]-[c^{-1}|ac|(ac)^{-1}]-[c^{-1}|bc|(bc)^{-1}]+[c^{-1}|c|c^{-1}]\\
&-[a^{-1}|c^{-1}|ac]+[(ab)^{-1}|c^{-1}|abc]-[c^{-1}|(ab)^{-1}|abc]+[c^{-1}|a^{-1}|ac]\\
&-[b^{-1}|c^{-1}|bc]+[c^{-1}|b^{-1}|bc]-[b^{-1}|a^{-1}|c^{-1}]+[b^{-1}|c^{-1}|a^{-1}]\\
&-[c^{-1}|b^{-1}|a^{-1}]-[c|(abc)^{-1}|abc]+[c|(ac)^{-1}|ac]+[c|(bc)^{-1}|bc]\\
&-[c|c^{-1}|c]+[a|c|(ac)^{-1}]-[ab|c|(abc)^{-1}]+[c|ab|(abc)^{-1}]\\
&-[c|a|(ac)^{-1}]+[b|c|(bc)^{-1}]-[c|b|(bc)^{-1}]+[b|a|c]-[b|c|a]\\
&+[c|b|a]\bigg\}\otimes 1.
\end{align*}
Combining all these we see that $d_{2,2}^2(\Lambda(a,b,c))$ is the following element of $E_{0,3}^2$
\begin{align*}
d_{2,2}^2(\Lambda(a,b,c))=&\bigg\{[c^{-1}|abc|(abc)^{-1}]-[c^{-1}|ac|(ac)^{-1}]-[c^{-1}|bc|(bc)^{-1}]+[c^{-1}|c|c^{-1}]\\
&-[a^{-1}|c^{-1}|ac]+[(ab)^{-1}|c^{-1}|abc]-[c^{-1}|(ab)^{-1}|abc]+[c^{-1}|a^{-1}|ac]\\
&-[b^{-1}|c^{-1}|bc]+[c^{-1}|b^{-1}|bc]-[b^{-1}|a^{-1}|c^{-1}]+[b^{-1}|c^{-1}|a^{-1}]\\
&-[c^{-1}|b^{-1}|a^{-1}]-[c|(abc)^{-1}|abc]+[c|(ac)^{-1}|ac]+[c|(bc)^{-1}|bc]\\
&+[a|c|(ac)^{-1}]-[ab|c|(abc)^{-1}]+[c|ab|(abc)^{-1}]-[c|a|(ac)^{-1}]\\
&+[b|c|(bc)^{-1}]-[c|b|(bc)^{-1}]+[b|a|c]-[b|c|a]+[c|b|a]\\
&+[abc|(abc)^{-1}|abc]-[ab|(ab)^{-1}|ab]-[bc|(bc)^{-1}|bc]\\
&-[ac|(ac)^{-1}|ac]+[a|a^{-1}|a]+[b|b^{-1}|b]\bigg\}\otimes 1
\end{align*}
By adding the null element
\begin{align*}
{}&d_4\bigg(\bigg\{-[c|c^{-1}|abc|(abc)^{-1}]+[c|c^{-1}|ac|(ac)^{-1}]+[c|c^{-1}|bc|(bc)^{-1}]-[c|c^{-1}|a^{-1}|ac]\\
&-[c|c^{-1}|c|c^{-1}]-[c|c^{-1}|b^{-1}|bc]+[c^{-1}|c|(abc)^{-1}|abc]+[ac|a^{-1}|c^{-1}|ac]\\
&-[abc|(ab)^{-1}|c^{-1}|abc]+[bc|b^{-1}|c^{-1}|bc]+[c|ab|(ab)^{-1}|ab]-[c|b|b^{-1}|b]-[c|a|a^{-1}|a]\\
&-[c|a|b|(ab)^{-1}]-[a|b|c|(abc)^{-1}]+[abc|b^{-1}|a^{-1}|c^{-1}]-[abc|b^{-1}|c^{-1}|a^{-1}]\\
&+[a|c|c^{-1}|a^{-1}]-[a|bc|b^{-1}|c^{-1}]+[a|c|b|b^{-1}]+[ac|b|b^{-1}|a^{-1}]-[a|bc|(bc)^{-1}|a^{-1}]\\
&-[b|c|(bc)^{-1}|a^{-1}]+[c|c^{-1}|b^{-1}|a^{-1}]-[c|c^{-1}|c|(abc)^{-1}]\bigg\}\otimes 1 \bigg)
\end{align*}
we see that, modulo $\text{im}(d_4)$,
\begin{align*}
d_{2,2}^2(\Lambda(a,b,c)) &=-([a|b|c]+[c|a|b]+[b|c|a]-[b|a|c]-[c|b|a]-[a|c|b])\otimes 1\\
&=-a\wedge b\wedge c.
%-{\bf c}\bigg( {\mtxx a 0 0 {a^{-1}}}, {\mtxx b 0 0 {b^{-1}}}, {\mtxx c 0 0 {c^{-1}}} \bigg)
\end{align*}
Thus we obtain the desired exact sequence
\[
\tors (\mu(A), \mu(A))^{\Sigma_2'}  \arr H_3(\SL_2(A),\z) \arr \RB(A) \arr 0.
\]
Now let $A$ be a domain. Since $\mu(A)$ is direct limit of finite cyclic  groups, then 
\[
\tors (\mu(A), \mu(A))^{\Sigma_2'} =\tors (\mu(A), \mu(A)).
\]
Let $F$ be the quotient field of $A$ and $\overline{F}$ the algebraic closure of $F$. It is very easy to 
see that $\RB(\overline{F})=\BB(\overline{F})$. The classical Bloch-Wigner exact sequence claims that the sequence
\[
0\arr \tors (\mu(\overline{F}), \mu(\overline{F}))  \arr H_3(\SL_2(\overline{F}),\z) \arr \BB(\overline{F}) \arr 0
\]
is exact. Now the final claim follows from the commutative diagram with exact rows
\[
\begin{tikzcd}
 &\tors (\mu(A), \mu(A))  \ar[r] \ar[d, hook] & H_3(\SL_2(A),\z) \ar[r] \ar[d] & \mathcal{RB}(A) \ar[r] \ar[d] & 0\\
	0 \ar[r] & \tors (\mu(\overline{F}), \mu(\overline{F}))  \ar[r] & H_3(\SL_2(\overline{F}),\z) \ar[r] & \BB(\overline{F}) \ar[r] & 0
\end{tikzcd}
\]
and the fact that the natural map $\tors (\mu(A), \mu(A)) \arr \tors (\mu(\overline{F}), \mu(\overline{F})) $ is injective.
\end{proof}

\begin{cor}
Let $A$ be a local domain of characteristic $2$, where its residue field has more than  $2^6$ elements. Then we have
 the refined Bloch-Wigner exact sequence
\[
0\arr \tors(\mu(A), \mu(A))  \arr H_3(\SL_2(A),\z) \arr \RB(A) \arr 0.
\]
\end{cor}
\begin{proof}
This follows from Proposition \ref{GE2C}, Proposition \ref{iso-hut} and  Theorem \ref{char=2}.
\end{proof}

%%%%%%%%%%%%%%%%%%%%%%%%%%%%%%%%%%%%%%%%%%%%%%%%%%%%%%%%

\end{document}